\documentclass[11pt]{amsart}
\usepackage{amssymb}

\usepackage{color}
\usepackage{graphicx}
\usepackage[vmargin=2.5cm, hmargin=2.5cm]{geometry}
\usepackage{enumerate}

\author{Abdallah Assi} 
\address{Universit\'e d'Angers, Math\'ematiques,
49045 Angers ceded 01, France}
\email{assi@univ-angers.fr} 

\thanks{The first author is partially supported by the project GDR CNRS 2945 and a GENIL-SSV 2014 grant}

\author{Pedro A. Garc\'{\i}a-S\'anchez}
\address{Departamento de \'Algebra, Universidad de Granada, E-18071 Granada, Espa\~na}
\email{pedro@ugr.es} \urladdr{www.ugr.es/local/pedro}

\thanks{The second author is supported by the projects MTM2010-15595, FQM-343,  FQM-5849, G\'eanpyl (FR n\textordmasculine 2963 du CNRS), and FEDER funds}

\title{On curves with one place at infinity}



\newtheorem{teorema}{Theorem}

\newtheorem{proposicion}[teorema]{Proposition}
\newtheorem{lema}[teorema]{Lemma}

\newtheorem{corolario}[teorema]{Corollary}

\theoremstyle{remark}

\newtheorem{nota}[teorema]{Remark}
\newtheorem{example}[teorema]{Example}
\newtheorem{exemples}[teorema]{Examples}

\newcommand{\KK}{{\mathbb K}}
\newcommand{\NN}{{\mathbb N}}
\newcommand{\PP}{{\bf P}}

\begin{document}
\maketitle
\begin{abstract}
Let $f$ be a plane curve. We give a procedure based on Abhyankar's approximate roots to detect if it has a single place at infinity, and if so construct its associated $\delta$-sequence, and consequently its value semigroup. Also for fixed genus (equivalently Frobenius number) we construct all $\delta$-sequences generating numerical semigroups with this given genus. For a $\delta$-sequence we present a procedure to construct all curves having this associated sequence. 

We also study the embeddings of such curves in the plane. In particular, we prove that polynomial curves might not have a unique embedding.
\end{abstract}

\section*{Introduction}

Let $\KK$ be an algebraically closed field of characteristic zero and let $f(x,y)=y^n+a_1(x)y^{n-1}+\dots+a_n(x)$ be a nonzero polynomial of $\KK[x][y]$. Assume, after possibly a change of variables, that $\deg_x(a_i(x))<i$ for all $1\leq i\leq n$. Write $f(x,y)=y^n+\sum_{i,j, i+j<n}c_{ij}x^iy^j$ and let $F(x,y,u)=y^n+\sum c_{ij}u^{n-i-j}x^iy^j\in\KK[u,x,y]$. Let $C$ be the curve $f=0$ in $\KK^2$. Then the projective curve $\bar{C}: F=0$ is the projective closure of $C$ in $\PP_K^2$. Furthermore, $p=(0,1,0)$ is the unique point at infinity of $\bar{C}$ and $F(u,1,y)$ is the local equation of $\bar{C}$ at  $p$. We say that $f$ has one place at infinity if $F(u,1,y)$ is analytically irreducible in $\KK[[u,y]]$. Curves with one place at infinity play an important role in affine geometry. In particular, it has been proved in \cite{ab-1} and \cite{a-m-2} that if $f$ has one place at infinity, then so is for $f-\lambda$ for all $\lambda\in\KK$. Also, we can associate with $f$ a numerical semigroup that has some good properties (it is free, and thus a complete intersection and symmetric). By using the arithmetic of the semigroup of a polynomial with one place at infinity and its approximate roots, S.S. Abhyankar and T.T. Moh proved that given two polynomials $x(t),y(t)\in\KK[t]$ with $t$-degrees $n> m$, if $\KK[x(t),y(t)]=\KK[t]$, then  $m$ divides $n$, showing that a coordinate of $\KK^2$ has a unique embedding in $\KK^ 2$. 

The study of polynomials with one place at infinity is also motivated by the plane Jacobian conjecture. Let $g\in\KK[x,y]$. This  conjecture says the following: if the Jacobian $J(f,g)$ is a nonzero constant, then $\KK[x,y]=\KK[f,g]$. If $\KK[x,y]=\KK[f,g]$, then $f$ is equivalent to a coordinate, in particular $f$ has one place at infinity. Hence the plane Jacobian conjecture is equivalent to the following: if $J(f,g)$ is a nonzero constant, then $f$ has one place at infinity. Despite a lot of activities, this conjecture is still open.

Let $f$ be as above and let $g$ be a polynomial of $\KK[x,y]$. Let $C_1$ be the affine curve $g=0$. We say that $C$ is isomorphic to $C_1$ if the ring of coordinates $\KK[C]$ and $\KK[C_1]$ are isomorphic. We say that $C$ is equivalent to $C_1$ if $f=\sigma(g)$ for some automorphism $\sigma$ of $\KK[x,y]$. It is natural to ask which isomorphic curves are equivalent. Curves with one place at infinity, with their good properties, offer a good setting where this question can be studied.

The main aim of this paper is to give an algorithmic approach to the study of curves with one place at infinity, together with an implementation in the semigroup package \texttt{GAP} (\cite{gap}). Our contributions in this direction are the following. For (2) and (3) we make extensive use of the concept of $\delta$-sequence (see Page \pageref{def-delta-seq}).

\begin{enumerate}

\item (Section \ref{sec3}) Given a polynomial $f\in\KK[x,y]$, decide if $f$ has one place at infinity and if yes, compute the set of its approximate roots and the generators of its semigroup. Our approach here is the irreducibility criterion given by S.S. Abhyankar in \cite{ab2}. This criterion is straightforward from the equation. It is based on the notion of generalized Newton polygons and it does not use neither the resolution of singularities nor the calculation of Puiseux series at infinity. We have implemented this procedure using the \texttt{numericalsgps} (\cite{numericalsgps}) \texttt{GAP} (\cite{gap}) package, and named it \texttt{SemigroupOfValuesOfPlaneCurveWithSinglePlaceAtInfinity}.

\item (Section \ref{sec4}) Given a sequence of integers in ${\mathbb N}$, decide if this sequence generates the semigroup of a polynomial at infinity, and if it is the case, calculate an equation of such a polynomial. Our equation allows in particular, using the notion of generalized Newton polygons, the calculation of all polynomials with the given semigroup. The \texttt{GAP} function implemented for this is \texttt{CurveAssociatedToDeltaSequence} (compare with \cite[Algorithm 1]{f-s}).

\item (Section \ref{sec5}) Given a positive integer $g$, compute the set of  semigroups of polynomials with one place at infinity with genus $g$ (hence with conductor $2g$). Note that if $\Gamma$ is such a semigroup and if $f$ is a polynomial with one place at infinity whose semigroup is $\Gamma$, then $g$ is the geometric genus of a nonsingular element of the pencil $(f-\lambda)_{\lambda\in\KK}$, and $2g$ is the rank of the $\KK$-vector space ${\KK[x,y]}/{(f_x,f_y)}$, where $f_x$ (respectively $f_y$) denotes the $x$-derivative (respectively the $y$-derivative) of $f$. The function \texttt{DeltaSequencesWithFrobeniusNumber} has been designed to do this. In \cite{f-s} there is an alternative procedure to compute all $\delta$-sequences with given genus.
\end{enumerate}

All these \texttt{GAP} functions will be included in the next release of the package \texttt{numericalsgps}.

In Section \ref{sec1} we recall the main properties of curves with one place at infinity and we give a proof of Abhyanlar-Moh Theorem: a rational nonsingular one place curve is equivalent to a coordinate of $\KK^2$. In Section \ref{sec2} we explicitly describe the automorphism that transforms the equation of the curve into a coordinate. The method of the proof  allows us to give a simple demonstration of the fact that two equivalent polynomials with one place at infinity have the same $\delta$-sequences. In the last section, we study the embedding of polynomial curves (rational curves with one place at infinity). In particular, we give an example of  a polynomial curve with two inequivalent embeddings in the affine plane.

\section{Semigroup of one place at infinity curves}\label{sec1}

Let $f(x,y)=y^n+a_1(x)y^{n-1}+\dots+a_n(x)$ be a nonzero polynomial of $\KK[x][y]$ and assume, after possibly a change of variables, that for all $1\leq i\leq n$, $\deg_xa_i(x)<i$. Assume that $f$ has one place at infinity. Given a polynomial $g\in\KK[x,y]$, we define the intersection of $f$ with $g$, denoted int$(f,g)$, to be the rank of the $\KK$-vector space ${{\KK[x,y]}/{(f,g)}}$.  Clearly $0=\mathrm{int}(f,1)$ and int$(f,g_1g_2)=\mathrm{int}(f,g_1)+\mathrm{int}(f,g_2)$. 

The set  $\{\mathrm{int}(f,g)\mid g\in \KK[x,y]\setminus (f)\}$ is a subsemigroup of $\NN$. We denote it by $\Gamma(f)$, and we say that $\Gamma(f)$ is the \emph{Abhyankar semigroup} associated to $f$. We also say that a subsemigroup $\Gamma$ of $\mathbb N$ is an Abhyankar semigroup if there exists $f$ such that $\Gamma=\Gamma(f)$ (for basic properties on numerical semigroups,  please, refer to \cite{ns}).

Let $d\in\NN^*$ and assume that $d\mid n$. Let $G$ be a monic polynomial of $\KK[x][y]$ of degree $\frac{n}{d}$ in $y$. Write
$$
f=G^d+\alpha_1(x,y)G^{d-1}+\dots+\alpha_d(x,y),
$$
where $\deg_y\alpha_k(x,y) < \frac{n}{d}$ for all $1\leq k\leq d$. We  say the $G$ is a \emph{$d$th approximate root} of $f$ if $\alpha_1(x,y)=0$. It is well known that a $d$th approximate root exists and it is unique (see \cite{ab-1}). We denote it by $\mathrm{App}_d(f)$.

Assume, after possibly a change of variables, that $a_1(x)=0$. In particular $\mathrm{App}_{n}(f)=y$. Let $r_0=d_1=n$ and  $r_1=m=\deg_xa_n(x)=\mathrm{int}(f,y)$. For all $k\geq 2$, let $d_{k}=\gcd(d_{k-1},r_{k-1})$, $g_k=\mathrm{App}_{d_{k}}(f)$  and $r_k=\mathrm{int}(f,g_k)$. It is well known that there exists $h\in\NN$ such that $d_{h+1}=1$ (see \cite{ab-1}). Thus every Abhyankar semigroup is a numerical semigroup, that is, its complement in $\mathbb N$ is finite. We set $g_1=y=\mathrm{App}_{d_1}(f)$ and $g_{h+1}=f$.  We also set $e_k=\frac{d_k}{d_{k+1}}$ for all $1\leq k\leq h$.

Recall that a numerical semigroup $\Gamma=\langle r_0,\ldots,r_h\rangle$ is a \emph{free} semigroup (for this arrangement of generators, and with the notation of the above paragraph) if $e_k r_k\in \langle r_0,\ldots, r_{k-1}\rangle$ for $k>0$ and $e_k$ is the minimum positive integer with this property; we are not imposing $\{r_0,\ldots,r_h\}$ to be a minimal generating system of $\Gamma$. These semigroups are a special kind of \emph{complete intersections} (semigroups with minimal presentations with the minimal possible cardinality: $h$ with the above notation). Complete intersection numerical semigroups are \emph{symmetric} (see for instance \cite[Chapter 8]{ns}), that is, for every $x\in \mathbb Z\setminus \Gamma$, $\mathrm F(\Gamma)-x\in S$, where $\mathrm F(\Gamma)$ the largest element in $\mathbb{Z}\setminus \Gamma$, called the \emph{Frobenius number} of $S$. When the semigroup is free and $r_0<\cdots <r_h$, we say that it is \emph{telescopic}.

\begin{lema}[\cite{ab-3}]\label{dev-appr}
Let $g\in \KK[x,y]$ and let $B=\left\lbrace \underline{\theta}=(\theta_1,\ldots,\theta_{h+1})\mid \theta_i<e_i \hbox{ for all }i\in\{1,\ldots,h\}\right\rbrace$. Then there exists $B'\subset B$ with finitely many elements such that 
\[g=\sum_{\underline{\theta}\in B'} c_{\underline{\theta}}(x) g_1^{\theta_1}\dots g_h^{\theta_h} f^{\theta_{h+1}}.\]
Moreover,
\begin{enumerate}[(i)]
\item $f\mid g$ if and only if $\theta_{h+1}>1$ for all $\theta\in B'$, 
\item if $f\nmid g$, then there exists a unique $\theta^0 \in B'$ such that $\theta_{h+1}^0=0$ and $\mathrm{int}(f,g)=\deg_xc_{\underline{\theta}^0}(x)r_0+\sum_{i=1}^h\theta_i^0 r_i= \max\{\deg_xc_{\underline{\theta}}(x)r_0+\sum_{i=1}^h\theta_i r_i \mid \underline{\theta}\in B' \}$.
\end{enumerate}
\end{lema}


\begin{proposicion} Under the standing hypothesis:
\begin{enumerate}[i)]
\item $\Gamma(f)=\langle r_0,r_1,\dots,r_h\rangle$;
\item for all $1\leq k\leq h$,  $e_kr_k\in \langle r_0,\dots,r_{k-1}\rangle$, that is,  $\Gamma(f)$ is free with respect to the arrangement $(r_0,r_1,\dots,r_h)$;
\item for all $2\leq k\leq h, r_{k-1}d_{k-1} >r_kd_k$;
\item the  Frobenius number of $\Gamma(f)$ is  \[\mathrm F(\Gamma(f))=\left(\sum_{k=1}^h(e_k-1)r_k\right)-r_0.\] 
\end{enumerate}
In particular, the conductor  of $\Gamma(f)$, is  $\mathrm C(\Gamma(f))=\left(\sum_{k=1}^h(e_k-1)r_k\right)-r_0+1$.
\end{proposicion}
\begin{proof} See \cite{ab-1, a-m-1, a-m-2}.
\end{proof}

Let $f_x,f_y$ denote the partial derivatives of $f$.  We have the following.
 
\begin{lema}\label{form-c}
Define $\mu(f)=\mathrm{rank}_{\KK}({\KK[x,y]}/ {(f_x,f_y)})$.
\begin{enumerate}[i)]
\item $\mu(f)=\mathrm C(\Gamma(f))=(\sum_{k=1}^h(e_k-1)r_k)-r_0+1$.
 
\item $\mathrm{int}(f,f_y)=\mu(f)+r_0-1=\sum_{k=1}^h(e_k-1)r_k$.
\end{enumerate}
\end{lema}
\begin{proof}
See \cite{ab-1,a-2}.
\end{proof}

\begin{proposicion}\label{recursive-appr}
Let the notations be as above. For all $1\leq k\leq h$, $g_k=\mathrm{App}_{d_k}(f)$ is a polynomial with one place at infinity. Furthermore, the following conditions hold.
\begin{enumerate}[i)]
\item $\Gamma(g_k)=\langle{r_0/ d_{k}},\dots,{r_{k-1}/ d_k}\rangle$.

\item  $\mu(f)=d_k\mu(g_k)+(\sum_{i=k}^h(e_i-1)r_i)-d_k+1$.

\item $\lbrace g_1,\dots,g_{k-1}\rbrace$ is the set of approximate roots of $g_k$.
\end{enumerate}
\end{proposicion}
\begin{proof}
See \cite{ab-1}.
\end{proof}
 
\begin{example}
Let $f=((y^3-x^2)^2-xy^2)^4-(y^3-x^2)$.
\begin{verbatim}
gap> f:=((y^3-x^2)^2-x*y^2)^4-(y^3-x^2);;
gap> SemigroupOfValuesOfPlaneCurveWithSinglePlaceAtInfinity(f,"all");
[ [ 24, 16, 28, 7 ], [ y, y^3-x^2, y^6-2*x^2*y^3+x^4-x*y^2 ] ]
gap> g:=last[2][3];
y^6-2*x^2*y^3+x^4-x*y^2
gap> SemigroupOfValuesOfPlaneCurveWithSinglePlaceAtInfinity(g,"all");
[ [ 6, 4, 7 ], [ y, y^3-x^2 ] ]
gap> s:=SemigroupOfValuesOfPlaneCurveWithSinglePlaceAtInfinity(f);
<Numerical semigroup with 4 generators>
gap> FrobeniusNumber(s);
57
\end{verbatim}
\end{example}
 
\begin{corolario}\label{muf-mugk} 
Let the notations be as above.  We have $\mu(f)=0$ if and only if  $\mu(g_k)=0$ and $r_i=d_{i+1}$ for all $k\leq i\leq h$.
\end{corolario}
\begin{proof}
For all $i\geq k$ we have $e_i>1$ and $r_i\geq d_{i+1}\geq 1$, hence  $(e_i-1)r_i\geq d_i-d_{i+1}$. In particular, $\left(\sum_{i=k}^h(e_i-1)r_i\right)-d_k+1\geq \left(\sum_{i=k}^h(d_i-d_{i+1})\right)-d_k+1=0$. If $\mu(f)=0$, then $\mu(g_k)=\left(\sum_{i=k}^h(e_i-1)r_i\right)-d_k+1=0$. For $k=h$, $(e_h-1)r_h-d_h+1=0=(d_h-1)(r_h-1)$. As $d_h\neq1$, $r_h=1=d_{h+1}$. Now for $k=h-1$, $(e_{h-1}-1)r_{h-1}+(e_h-1)r_h-d_{h-1}+1=0$. We use the case $k=h$, and we obtain $(e_{h-1}-1)r_{h-1}-(d_{h-1}-d_h)$. But $e_{h-1}-1=(d_{h-1}-d_h)/d_h$, and so $(d_{h-1}-d_h)\left( \frac{r_{h-1}}{d_h} -1\right)=0$. We conclude $r_{h-1}=d_{h}$. And we keep reasoning in this way, and we finally derive $r_i=\displaystyle{d_{i+1}}$ for all $i\in \{k,\dots,h\}$. 

The converse is obvious.
\end{proof}

\begin{corolario} \label{cor-mu-f-0}
Let the notations be as above.  The following conditions are equivalent.
\begin{enumerate}[(i)]
\item $\mu(f)=0$.
\item $r_i=\displaystyle{d_{i+1}}$ for all $1\leq i\leq h$.
\item  $0=\mathrm C(\Gamma(f))$.
\item There exist $x(t),y(t)\in\KK[t]$ such that $f(x(t),y(t))=0$ and  $\KK[x(t),y(t)]=\KK[t]$.
\end{enumerate}
Furthermore, each of the conditions above implies that $m$ divides $n$.
\end{corolario}
  
\begin{proof} \emph{(i) equivalent to (ii)} follows from  Corollary \ref{muf-mugk}.

\emph{(i) implies (iii)} is  obvious by Lemma \ref{form-c}.

\emph{(iii) implies (iv)} Since $\mu(f)=0$, the genus of $\Gamma(f)=0$. This forces the genus of the curve $C=\mathrm V(f)$ to be zero and consequently $f$ is rational. 
Also, since $f$ has one place at infinity, $f$ can be parametrized by polynomials (see \cite{ab-n}). Now $1\in\Gamma(f)$, and so there exists $p\in K[x,y]$ such that $\mathrm{rank}_\KK (\KK[x,y]/(f,p))=1=\deg_t p(x(t),y(t))$ (one can prove this by using the isomorphism between $\KK[x,y]/(f)$ and $\KK[(x(y),y(t)]$).  If follows that $at+b=p(x(t),y(t))$ for some $a\in\KK^*$, $b\in\KK$. Hence $\KK[at+b]=\KK[t]\subseteq \KK[x(t),y(t)]\subseteq K[t]$. In particular $\KK[t]= \KK[x(t),y(t)]$.

\emph{(iv) implies (iii)} $t=p(x(t),y(t))$ for some $p(x,y)\in\KK[x,y]$, whence $1\in\Gamma(f)$.

\noindent Now assume that Condition (ii) holds. We have $r_1=m=d_2$. Thus $m$ divides $n$.
\end{proof}

\section{Reduced forms of one place curves}\label{sec2}

Let  $f=y^n+a_2(x)y^{n-2}+\dots+a_n(x)$ be a polynomial with one place at infinity and assume that $n > m=\deg_xa_n(x)$. In particular this implies that $\deg_x a_i(x)<i$ for all $i\in\{1,\ldots,n\}$, because otherwise the Newton polygon associated to $f$ would have more than one edge, and consequently $f$ would have more than one place at infinity (see \cite[Chapter 4]{ab-3}).

Let $\underline{r}=(r_0,\dots,r_h), \underline{d}=(d_1,\dots,d_h, d_{h+1}=1), \underline{e}=(e_1,\dots,e_h)$ be the set of characteristic sequences associated with $f$ as in Section \ref{sec1},  and let $g_1=y,g_2,\dots,g_h,g_{h+1}=f$ be the set of approximate roots of $f$.

\begin{nota}[Generalized Newton polygons, see \cite{ab2}]\label{nota21} Let $1\leq k\leq h$ and let $g_{k+1}=g_k^{e_k}+\alpha_2(x,y)g_k^{e_k-2}+\dots+\alpha_{e_k}(x,y)$ be the expansion of $g_{k+1}$ with respect to $g_k$. We have the following:
\begin{enumerate}[(i)]
\item $\mathrm{int}(g_{k+1},\alpha_{e_k}(x,y))=\frac{r_k}{ d_{k+1}}e_k= \mathrm{int}(g_{k+1},g_k^{e_k})$.

\item For all $i=2,\dots,e_{k}-1$, $\mathrm{int}(g_{k+1},\alpha_{i}(x,y)) < i\frac{r_k}{ d_{k+1}}$. 

\end{enumerate}
Since the expression of $g_{k+1}$ in terms of $g_k$ is monic and there is no term of degree $e_{k}-1$, we set $\alpha_0=1$ and $\alpha_1=0$. Then the convex hull of the points $(0,0), \left(\mathrm{int}(g_{k+1},\alpha_{i}(x,y)),\frac{r_k}{d_{k+1}}(e_k-i)\right)$, $i\in\{0,\dots,e_k\}$ is the triangle with vertices $(0,0),\left(0,\frac{r_k}{d_{k+1}}e_k\right)$ and ${(\frac{r_k}{d_{k+1}}e_k,0)}$. The \emph{generalized Newton polygon} of $g_{k+1}$ with respect to $g_k$ is the segment of line joining $\left(0,\frac{r_k}{d_{k+1}}e_k\right)$ with $\left(\frac{r_k}{d_{k+1}}e_k,0\right)$ (recall that, by (ii), it does not contain any of the  points  $\left(\mathrm{int}(g_{k+1},\alpha_{i}(x,y)),(e_k-i)\frac{r_k}{d_{k+1}}\right), i=2,\dots,e_k-1$).
\end{nota}

\begin{lema} \label{muf0}
Let the notations be as above. If $\mu(f)=0$, then for all $k\in\{1,\dots,h\}$, there exist $a^k_2,\dots,a^k_{e_k},a_k\in \KK$ such that $a^k_{e_k}\not=0$ and 

\begin{enumerate}[(i)]
\item $g_2=g_1^{e_1}+a^1_2g_1^{e_1-2}+\dots+a^1_{e_1-1}g_1+a^1_{e_1}x+a_1$.

\item  If $k\geq 2$, then $g_{k+1}=g_k^{e_k}+a^k_2g_k^{e_k-2}+\dots+a^k_{e_k-1}g_k+a^k_{e_k}g_{k-1}+a_k$.
\end{enumerate}
\end{lema}

\begin{proof} We prove (ii)  for $k=h$. Let, to this end,
$$
f=g_{h+1}=g_h^{d_h}+a_2(x,y)g_h^{d_h-2}+\dots+a_{d_h-1}(x,y)g_h+a_{d_h}(x,y)
$$
be the expansion of $f$ with respect to  $g_h$. From this expression of $f$, we obtain $\mathrm{int}(f,g_h)=\mathrm{int}(g_h,a_{d_h}(x,y))$. By Remark \ref{nota21},  $\mathrm{int}(f,a_{d_h}(x,y))=\mathrm{int}(f,g_h^{d_h})=r_hd_h=d_h$ ($r_h=d_{h+1}=1$ by Corollary \ref{cor-mu-f-0}) and for all $k=2,\dots,d_h-1$, $\mathrm{int}(f,a_k(x,y)) < r_hk= k < d_h$. From $\mathrm{int}(f,g_h^{d_h})=r_hd_h$, we deduce $\mathrm{int}(g_h,a_{d_h}(x,y))=r_h=d_{h+1}=1$. But for all $k\in \{2,\dots,d_h\}$, $\mathrm{int}(f,a_k(x,y))\in \langle r_0,r_1,\dots,r_{h-1}\rangle$ and $d_h=\gcd(r_0,\dots,r_{h-1})$. This proves that $\mathrm{int}(f,a_k(x,y))=0$ for all $k$, and consequently $a_2(x,y),\dots,a_{d_h-1}(x,y)\in\KK$. 
We know that $\mathrm{int}(f,a_{d_h}(x,y))=d_h$, which by Corollary \ref{cor-mu-f-0} equals $r_{h-1}$. 
From Lemma \ref{dev-appr}, $a_{d_h}=\sum_{\underline{\theta}\in B'}  c_{\underline{\theta}}(x) g_1^{\theta_1}\dots g_h^{\theta_h} f^{\theta_{h+1}}$. Since $\deg_y a_{d_h}< \deg_y g_h$, this expression is of the form $a_{d_h}=\sum_{\underline{\theta}\in B'}  c_{\underline{\theta}}(x) g_1^{\theta_1}\dots g_{h-1}^{\theta_{h-1}}$. Moreover, Lemma \ref{dev-appr} also tells us that there exists a unique $\underline{\theta}^0$ such that $r_{h-1}=\mathrm{int}(f,a_{d_h})=\deg_x c_{\underline{\theta}^0}(x)r_0+\sum_{i=1}^{h-1}\theta_i^0 r_i$. This forces $\deg_x c_{\underline{\theta}^0}(x)= \theta_1^0=\dots =\theta_{h-2}^0=0$ and $\theta_{h-1}^0=1$. Observe also that $r_{h-1}=d_h\le r_i$ for all $i\in\{0,\ldots, h-2\}$. We deduce that there is $a_{e_h}^h\in \KK^*$ and $a_h\in\KK$ such that $a_{d_h}(x,y)=a_{e_h}^hg_{h-1}+a_h$. This proves our assertion. Now the same argument works for all $2\leq k\leq h-1$.

Let $k=1$. We have $\Gamma(g_2)=\langle \frac{r_0}{ d_2}=e_1, \frac{r_1}{d_2}\rangle$. Furthermore,
$$
g_2=y^{e_1}+a_2(x)y^{e_1-2}+\dots+a_{e_1}(x).
$$
Now $\mathrm{int}(g_2,a_{e_1}(x))= \mathrm{int}(g_2,y^{e_1})= e_1\frac{r_1}{d_2}$ (Remark \ref{nota21} once more). Also, by definition $\mathrm{int}(g_2,a_{e_1}(x)=\mathrm{rank}_\KK(\KK[x,y]/(g_2,a_{e_1}(x))) =e_1\deg_x a_{e_1}(x) = \frac{r_0}{d_2}\deg_x a_{e_1}(x)$. Since $r_0=d_1$ and $r_1=d_2$, we get $\deg_x a_{e_1}(x)=1$ and there is $a_{e_1}^1\in \KK^*$ and $a_1\in\KK$ such that $a_{e_1}(x)=a_{e_1}^1x+a_1$. This proves (i).
\end{proof}

\begin{lema}\label{Kxy} With the same hypothesis as in  Lemma \ref{muf0}, we have
\[\KK[f,g_h]=\KK[g_h,g_{h-1}]=\dots=\KK[g_2,g_1]=\KK[x,y].\]
\end{lema}
\begin{proof} We have $\KK[x,y]=\KK[g_1,g_2]=\KK[y,g_2]$. In fact $g_2=g_1^{e_1}+a^1_2g_1^{e_1-2}+\dots+a^1_{e_1-1}g_1+a^1_{e_1}x+a_1$, where $a^1_2,\dots,a^1_{e_1-1},a_1\in \KK$ and $a^1_{e_1}\in \KK^*$. Since $g_1=y$,  the result is obvious. Let $k\in \lbrace 2,\dots,h\rbrace$. As $g_{k+1}=g_k^{e_k}+a^k_2g_k^{e_k-2}+\dots+a^k_{e_k-1}g_k+a^k_{e_k}g_{k-1}+a_k$, where  $a^k_2,\dots,a^k_{e_k},a_k\in \KK$ and $a^k_{e_k}\in\KK^*$, we have  $\KK[g_{k},g_{k+1}]=\KK[g_{k-1},g_k]$. Now our assertion results by induction on $k=1,\dots,h$.
\end{proof}

\begin{corolario}\label{corKxy} Let the hypotheses be as in  Lemma \ref{muf0}. Define $\sigma_0:\KK[g_1=y,g_2]\to \KK[x,y], \sigma_0(g_1)=x, \sigma_0(g_2)=y$ and for all $1\leq k\leq h-1$, let $\sigma_{k}:\KK[g_{k+1},g_{k+2}]\to \KK[g_{k},g_{k+1}], \sigma_k(g_{k+1})=g_{k}, \sigma_k(g_{k+2})=g_{k+1}$.
\begin{enumerate}[i)]
\item For all $0\leq k\leq h-1$, $\sigma_k$ is an automorphism.

\item Let $\sigma^h=\sigma_{0}\circ\sigma_{1}\circ\dots\circ\sigma_{h-1}:\KK[g_h,f=g_{h+1}]\to \KK[x,y]$, then $\sigma^h$ is an automorphism such that $\sigma^h(f)=y$ (and $\sigma^h(g_h)=x$).
\end{enumerate}
\end{corolario} 
\begin{proof}
This results from  Lemmas \ref{muf0} and \ref{Kxy}. 
\end{proof}

Let the notations be as above. In particular $f=y^n+a_2(x)y^{n-2}+\dots+a_n(x)$ and $r_0=n >r_1=m=\deg_xa_n(x)$. Let $k$ be the greatest element in $\lbrace 1,\dots,h+1\rbrace$ such that $\mu(g_{k})=0$ (observe that $\mu(g_1)=0$). By Corollaries \ref{muf-mugk} and \ref{cor-mu-f-0} and Proposition \ref{recursive-appr}, $\mu(g_1)=\dots=\mu(g_k)=0$ and $r_i=d_{i+1}$ for all $1\leq i\leq k-1$. With the notations of Corollary \ref{corKxy}, $\sigma^{k-1}:\KK[g_{k-1},g_k]\to \KK[x,y]$ is an automorphism such that $\sigma^{k-1}(g_{k-1})=x$ and $\sigma^{k-1}(g_k)=y$. Also, $\mu(g_{i})>0$ for all $k+1\leq i\leq h+1$. Let $\tilde{f}=\sigma^{k-1}(f)$. Assume that $k\ge 2$, and let us focus on $k=2$. Then
$$
g_3=g_2^{e_2}+\alpha^2_2(x,y)g_2^{e_2-2}+\dots+\alpha^2_{e_2}(x,y).
$$
The conductor $\mu(g_2)$ of $\Gamma(g_2)$ is $0$, and Lemma \ref{muf0} states that $g_2=y^{e_1}+a^1_2y^{e_1-2}+\dots+a^1_{e_1-1}y+a^1_{e_1}x+a_1$, with $a_{e_1}^1\neq 0$. Set $\varphi(y)=y^{e_1}+a^1_2y^{e_1-2}+\dots+a^1_{e_1-1}y+a_1$, and $a=a_{e_1}^1$. Hence $x=\frac{g_2-\varphi(y)}{a}$. With the notations of  Corollary \ref{corKxy}, we have $\sigma_0(g_2)=y$ and $\sigma_0(g_1)=x$. Thus
$$
\sigma_0(g_3)=y^{e_2}+\alpha^2_2\left(\frac{y-\varphi(x)}{a},x\right)y^{e_2-2}+\dots+\alpha^2_{e_2}\left(\frac{{y-\varphi(x)}}{a},x\right).
$$
Note that by using Remark \ref{nota21} and Lemma \ref{dev-appr}, we deduce that there exist $c,d\in\NN$, with $d< \frac{n}{d_2}$, such that int$(g_3,\alpha^2_{e_2}(x,y))=\mathrm{int}(g_3,g_2^{e_2})=e_2\frac{r_2}{d_3} =e_2\mathrm{int}(g_2,x^cy^d)=e_2\left(c\frac{n}{d_2}+d\frac{r_1}{d_2}\right)=e_2\left( c\frac{n}{d_2}+d\right)$ (because $r_1=d_2$). 

Also, from Remark \ref{nota21}, we know that $\mathrm{int}(g_3,\alpha_k^2(x,y))< k\frac{r_2}{d_3}$.

But 
\[a^{c}\sigma_0(x^cy^d)= a^{c}\left(\frac{{y-\varphi(x)}}{a}\right)^cx^d=x^{c\frac{n}{d_2}+d}+\sum_{\begin{matrix}j<e_2,\\ i\frac{d_2}{d_3}+j\frac{r_2}{d_3}<\frac{d_2}{d_3}\frac{r_2}{d_3}\end{matrix}}c_{ij}x^iy^j,\] for some $c_{ij}\in \KK$. Hence, up to constants, $\sigma_0(x^cy^d)=x^{r_2/d_3}+\sum_{j<e_2,i\frac{d_2}{d_2}+j\frac{r_2}{d_3}<\frac{d_2}{d_3}\frac{r_2}{d_3}}x^iy^j$. The same calculations  with $\sigma_0(\alpha_2^2(x,y)), \ldots, \sigma_0(\alpha_{e_2-1}^2(x,y))$ shows that
$$
\sigma_0(g_3)=y^{e_2}+\bar{\alpha}^2_2(x)y^{e_2-2}+\dots+\bar{\alpha}^2_{e_2}(x)+\bar{c}x^{\frac{r_2}{ d_3}},
$$
with $\bar{c}\in\KK^*$, and $\Gamma(g_3)=\langle e_2=\frac{d_2}{d_3},\frac{r_2}{d_3}\rangle=\langle \frac{r_1}{d_3},\frac{r_2}{d_3}\rangle$. We prove in a similar way that $\Gamma(\sigma_0(f))=\langle r_1,r_2,\dots,r_h\rangle$ and that $\bar{g}_1=\sigma_0(g_2)=y,\bar{g}_2=\sigma_0({g_3}),\dots,\bar{g}_{h-1}=\sigma_0(g_h)$ are the set of approximate roots of $\bar{f}=\sigma_0(f)$. If $k\geq 3$, then $\mu(\bar{g}_2)=0$, whence we restart the process above with $\bar{g}_2$.  We finally get the following:

\begin{enumerate}
\item ${\sigma}^{k-1}(f)=y^{d_k}+\bar{\alpha}_2(x,y)y^{d_k-2}+\dots+\bar{\alpha}_{d_k}(x,y)$.
\item ${\sigma}^{k-1}(g_k)=y,{\sigma}^{k-1}(g_{k+1}),\dots,{\sigma}^{k-1}(g_{h})$ are the approximate roots of ${\sigma}^{k-1}(f)$.
\item $\Gamma({\sigma}^{k-1}(f))=\langle r_{k-1}=d_k,\dots,r_h\rangle$.
\end{enumerate}

Note that if $\deg_y\sigma^{k-1}(f)=r_{k-1}=d_k< \mathrm{deg}_x(\sigma^{k-1}(f))=r_k$, then  the change of variables $X=y,Y=x$ will change the sequence $(r_{k-1},r_k,\dots,r_h)$ into $(r_k,r_{k-1},\dots,r_h)$ in such a way that  in the process above, we can always assume that the degree of ${\sigma}^{k-1}(f)$ in $y$ is also the total degree of $f$.

Observe also that $\mu({\sigma}^{k-1}(g_{k+1}))>0$. Hence if we define the \emph{reduced degree} of $f$, denoted $\mathrm{rdeg}(f)$, to be $\mathrm {rdeg}(f)=\inf\lbrace \deg(\sigma(f)) \mid  \sigma$ is an automorphism of $\KK[x,y]\rbrace$, then $\mathrm{rdeg}(f)=d_k$.

\subsection{The reduced equation} Let $F(x,y)=y^N+a_1(x)y^{N-1}+\dots+a_N(x)$ be a polynomial with one place at infinity and assume that $a_1(x)=0$ and also that $N> r_1=\deg_xa_1(x)> \gcd(N,r_1)$. Let $a_n(x)=cx^{r_1}+\sum_{i=1}^{r_1-1}c_ix^i$. We may assume, without loss of generality, that $c=1$. Write $F(x,y)= x^{r_1}+b_1(y)x^{r_1-1}+\dots+b_{r_1}(y)$. Suppose that $b_1(y)\not=0$ and let $p=\deg_yb_1(y)$. Clearly $p+r_1-1\leq N-1$, hence $p\leq N-r_1$. Furthermore, $pr_1+(r_1-1)N\leq Nr_1$, and thus $pr_1\leq N$. But $r_1$ does not divide $N$, whence $pr_1\leq N-1$. By the change of variables $X=x+\frac{b_1(y)}{r_1}, Y=y$, we get $\bar{F}(X,Y)=Y^{N}+\bar{a}_1(X)Y^{N-1}+\dots+\bar{a}_N(X)$ and  $\bar{a}_1(X)=\bar{a}_1\in\KK$. If $\bar{a}_1\not= 0$, then we consider the change of variables $X_1=X, Y_1=Y+{\frac{\bar{a}_1}{N}}$. Let $\tilde{F}(X_1,Y_1)=\bar{F}(X_1, Y_1-{\frac{\bar{a}_1}{N}})$ and let $\mathrm R(F)=\tilde{F}(x,y)$. We say that $\mathrm R(F)$ is the \emph{reduced form} of $F$. We have the following. 
\begin{itemize}

\item $\mathrm R(F)(x,y)=y^N+c_2(x)y^{N-2}+\dots+c_N(x)$ with $\deg_xc_i(x)<i$ for all $2\leq i\leq N$.

\item $c_N(x)=x^{r_1}+\sum_{i=1}^{r_1-2}d_ix^i$, and the coefficient of $x^{r_1-1}$ in $\mathrm R(F)$ is $0$.

\item Let $d$ be a divisor of $N$ and let $G_d=\mathrm{App}_d(\mathrm R(F))$. Then $\mathrm R(G_d)=G_d$.

\item Let $w_1$ (respectively $w_2$) be an $N$th root (respectively an $r_1$th root) of unity in $\KK$ and let $\sigma$ be the automorphism of $\KK[x,y]$ such that $\sigma(x)=w_1x,\sigma(y)=w_2y$, then $\sigma(\mathrm R(F))$ satisfies the same properties as $\mathrm R(F)$. Hence $\mathrm R(F)$ is unique modulo this type of automorphisms.
\end{itemize}

Let $f$ be as above. We define the \emph{reduced equation} of $f$, denoted $\mathrm r(f)$, to be $\mathrm R({\sigma}^{k-1}(f))$. It follows that $\mathrm{rdeg}(f)=\deg_y{\sigma}^{k-1}(f)=\deg_y(\mathrm r(f))$ (as above, $k$ is the greatest element in $\lbrace 1,\dots,h+1\rbrace$ such that $\mu(g_{k})=0$).

\begin{example}
Let $f(x,y)=(y^2-x)^2-xy$. We have $\Gamma(f)=\langle 4,2,3\rangle$, $g_1=y$, $g_2=y^2-x$. Now $\mu(g_2)=0$, hence 
$\sigma_0:\KK[g_2,y]\to \KK[y,x]$, $\sigma_0(g_2)=y$, $\sigma_0(y)=x$ is an automorphism of $\KK[x,y]$ and 
$\sigma_0(x)=\sigma_0(y)^ 2-\sigma_0(g_2)=x^2-y$. Therefore
$$
f^1=\sigma_0(f)=y^2-(x^2-y)x=y^2+xy-x^3.
$$
\begin{enumerate}
\item  $\deg_yf^1=2<3$, hence we interchange $x,y$ in $-f^1$, so we get $f^2(x,y)=y^3-x^2-xy$. 

\item  In $f^2$ We change $x$ into $wx$, where $w$ is a square root of $-1$. Hence we get $f^3(x,y)=y^3+x^2-wxy$. 

\item Let $X=x-\frac{wy}{2}$, $Y=y$. We have $f'(X,Y)=Y^3+X^2+\frac{Y^2}{4}$, hence $f^4(x,y)= y^3+x^2+\frac{y^2}{4}$.

\item Let $Y_1=y+\frac{1}{12}$, $X_1=x$. Then $f^4(X_1,Y_1-\frac{1}{12})= Y_1^3-\frac{1}{48}Y_1+\frac{2}{12^3}+X_1^2$. Finally $\mathrm r(f)(x,y)=y^3+x^2-\frac{1}{48}y+\frac{2}{12^3}$.
\end{enumerate}
\begin{verbatim}
gap> f:=(y^2-x)^2-x*y;;
gap> SemigroupOfValuesOfPlaneCurveWithSinglePlaceAtInfinity(f,"all");
[ [ 4, 2, 3 ], [ y, y^2-x ] ]
gap> rf:=y^3+x^2-1/48*y+2/12^3;;
gap> SemigroupOfValuesOfPlaneCurveWithSinglePlaceAtInfinity(rf,"all");
[ [ 3, 2 ], [ y ] ]
\end{verbatim}
\end{example}

Let $a_0,a_1,\dots,a_s$ be a set of coprime nonnegative integers and let $D_1=a_0$ and $D_k=\gcd(D_{k-1},a_{k-1})$ for all $2\leq k\leq s+1$. We say that $(a_0,\dots,a_s)$ is a \emph{$\delta$-sequence} if the following conditions hold:\label{def-delta-seq}

\begin{enumerate}[1)]
\item $\Gamma=\langle a_0,\dots,a_s\rangle$ is free with respect to the arrangement $(a_0,\dots,a_s)$.

\item For all $1\leq k\leq s-1$, $a_kD_k > a_{k+1}D_{k+1}$.

\item $a_0>a_1>D_2>D_3>\dots>D_{h+1}=1$.
\end{enumerate}
Note that Conditions 1) and 2) imply that $\Gamma$ is the semigroup of a polynomial with one place at infinity. 

\begin{proposicion}
Let the notations be as above. 
\begin{enumerate}
\item $\Gamma(\mathrm r(f))=\langle r_k,\dots,r_h\rangle$ and  $(r_k,\dots,r_h)$ is a $\delta$-sequence.
\item $\mathrm{rdeg}(f)=d_k=\deg_y\mathrm r(f)$.
\end{enumerate}
\end{proposicion}

\begin{proof} 
Obvious.
\end{proof}

In this reduction process the Abhyankar semigroup remains the same, but the $\delta$-sequences shorten.

Let $g,h$ be two polynomials of $\KK[x,y]$. We say that $g$ and $h$ are \emph{equivalent} if $h=\sigma(g)$ for some automorphism $\sigma$ of $\KK[x,y]$. 

\begin{teorema}\label{th-equiv}
Let the notations be as above and let $\Gamma(f)=\Gamma(\mathrm r(f))=\langle n=r_0,\dots,r_h\rangle$, and assume that $(r_0,\dots,r_h)$ is a $\delta$-sequence. Clearly $f$ is equivalent to $\mathrm r(f)$. Let $g$ be a polynomial with one place at infinity. If $g$ is equivalent to $f$, then $\Gamma(g)=\Gamma(\mathrm r(g))=\langle r_0,\dots,r_h\rangle$.
\end{teorema}

\begin{proof}
If $g$ is equivalent to $f$, then $\mathrm r(g)$ is also equivalent to $\mathrm r(f)$. Hence $\mathrm r(g)=\sigma(\mathrm r(f))$ for some automorphism $\sigma$ of $\KK[x,y]$. Write $\mathrm r(f)=y^n+a_2(x)y^{n-2}+\dots+a_n(x)$ with $\deg_xa_n(x)=r_1<n$ and let $\bar{\sigma}$ be an elementary automorphism of $\KK[x,y]$. Let $\bar{f}=\bar{\sigma}(\mathrm r(f))$.

\begin{enumerate}

\item If $\bar{\sigma}(y)=ay+b$ with $a,b\in\KK^*$, then $\bar{f}=a^ny^n+a^{n-1}by^{n-1}+\bar{a}_2(x)y^{n-2}+\dots+\bar{a}_n(x)$.

\item  If $\bar{\sigma}(y)=ay+h(x)$ with $a\in\KK^*$ and $\deg_xh(x)\geq 1$, then $\deg_y\bar{f}=n < \bar{r}_1=\deg_x\bar{f}$ and $n$ divides $\bar{r}_1$.

\item  If $\bar{\sigma}(y)=h(y)+ax$ with $a\in\KK^*$  and $\deg_yh(y)=p\geq 2$, then $\bar{r}_1=\deg_x\bar{f}=n$ divides  $\deg_y\bar{f}=pn$.
\end{enumerate} 

In all cases, $\bar{f}$ is not the reduced equation of a polynomial with one place at infinity. Furthermore, the total degree of $\bar{f}$ is $\geq n$. Since $\sigma$ is a composition of a finite number of elementary automorphisms, we get the same conclusion if either $\deg_y\sigma(y)>1$ or $\sigma(y)=ay+h(x)$ with $a\not=0$ and $h(x)\notin\KK^*$. Finally, $\sigma(y)=ay$, $a\in\KK^*$ and $\sigma(x)=bx+h(y)$ with $b\not=0$ and either $h(y)=0$ or $\deg_yh(y)r_1<n$. But then the coefficient of $x^{r_1-1}$ in $\sigma(\mathrm r(f))$ is non zero. Hence $\sigma(\mathrm r(f))\not=\mathrm r(g)$. Thus,  $\sigma(y)=ay, \sigma(x)=bx$ with $a,b\in\KK^*$. By definition of the reduced equation, $a$ (respectively $b$) is an $n$th (respectively an $r_1$th) root of unity. This proves our assertion.
\end{proof}

\section{The irreducibility criterion}\label{sec3}

Let $f=y^n+a_1(x)y^{n-1}+\dots+a_n(x)$ be a nonzero polynomial of $\KK[x,y]$ and assume, after possibly a change of variables, that $a_1(x)=0$ and also that $\deg_xa_i(x)<i$ for all $2\leq i\leq n$. Let $d_1=n>d_2>\dots>d_{h+1}=1$ be a set of divisors of $n$ and let $\underline{G}=(G_1,G_2,\dots,G_{h+1}=f)$ be a set of polynomials of degrees $\frac{n}{d_1},\dots,\frac{n}{d_h}$, respectively. 

Let $\underline{r}=(r_0=n,r_1,\dots,r_h)$ such that  $d_k=\gcd(r_0,\dots,r_{k-1})$ for all $1\leq k\leq h+1$ and let $e_k=\frac{d_k}{d_{k+1}}$ for all $1\leq k\leq h$. Let $B=\lbrace \underline{\theta}=(\theta_1,\dots,\theta_h,\theta_{h+1})\mid \hbox{for all } 1\leq k\leq h, 0\leq \theta_k< e_k\rbrace$. Given  $\underline{\theta}\in B$, we associate with $x^{\theta_0}\underline{G}^{\underline{\theta}}= x^{\theta_0}G_1^{\theta_1}\dots G_h^{\theta_h}$ the number $\mathrm{fint}(f, x^{\theta_0}\underline{G}^{\underline{\theta}},\underline{r})=\theta_0r_0+r_1\theta_1+\dots+r_h\theta_h= \underline{\theta}\cdot \underline{r}$ (dot product). Given a nonzero element $c_{\underline{\theta}}(x)\underline{G}^{\underline{\theta}}$, we set 
$\mathrm{fint}(f,c_{\underline{\theta}}(x) \underline{G}^{\underline{\theta}},\underline{r})=\mathrm{fint}(f, x^{\theta_0}\underline{G}^{\underline{\theta}},\underline{r})$, where $\theta_0=\deg_x(c_{\underline{\theta}}(x))$. Let $\alpha(x,y)$ be a nonzero polynomial of $\KK[x,y]$ and assume that $\deg_y\alpha(x,y)<n$. Write:
$$
\alpha(x,y)=\sum_{\underline{\theta}\in B}c_{\underline{\theta}}(x)G_1^{\theta_1}\cdots G_h^{\theta_h}.
$$
We set $\mathrm{fint}(f,\alpha,\underline{r})=\max\lbrace \mathrm{fint}(f,c_{\underline{\theta}}(x)G_1^{\theta_1}\dots G_h^{\theta_h})|c_{\underline{\theta}}(x)\not=0\rbrace$. There is a unique monomial of $\alpha$, say $c_{\underline{\theta^0}}(x) G_1^{\theta^0_1} \cdots G_h^{\theta^0_h}$, such that $\mathrm{fint}(f,\alpha,\underline{r})=\mathrm{fint}(f,c_{\underline{\theta^0}}(x)G_1^{\theta^0_1}\cdots G_h^{\theta^0_h})$. 

Let
$$
f=g_h^{d_h}+\alpha_1g_h^{d_h-1}+\dots+\alpha_{d_h}
$$
be the expansion of $f$ with respect to $g_h$. We say that $f$ is \emph{straight} with respect to $(g_h,\underline{r})$ if the following conditions hold:
\begin{enumerate}

\item $\mathrm{fint}(f,\alpha_{d_h},\underline{r})=r_hd_h$,

\item $\mathrm{fint}(f,\alpha_i,\underline{r}))< id_h$ for all $1\leq i\leq d_h-1$.
\end{enumerate}

\subsection{The criterion \cite{ab2}} 
Let $f$ be as above. Let $r_0=n=d_1$. If $a_n(x)=0$, then $y$ divides $f$, and thus $f$ has at least two places at infinity. Suppose that $a_n(x)\not=0$. Set $g_1=y$, $r_1=\deg_xa_n(x)$, $d_2=\gcd(r_0,r_1)$, and $\underline{r}^1=\left(\frac{r_0}{d_2},\frac{r_1}{d_2}\right)$. Let $g_2=\mathrm{App}_{d_2}(f)$ and let
$$
f=g_2^{d_2}+\alpha^2_2g_2^{d_2-2}+\dots+\alpha^2_{d_2}
$$

\noindent be the expansion of $f$ with respect to $g_2$. We set $r_2=\mathrm{fint}(g_2,\alpha^2_{d_2},\underline{r}^1)$ and $d_3=\gcd(r_2,d_2)$. We now restart with $g_3=\mathrm{App}_{d_3}(f)$, and so on. If $d_i=d_{i+1}$ for some $i$, then $f$ has at least two places at infinity. Suppose that $d_1>d_2>\dots$. There exists $h\geq 1$ such that $d_{h+1}=1$. For all $2\leq i\leq h+1$ we set $\underline{r}^{i-1}=\left(\frac{r_0}{d_i},\frac{r_1}{d_i},\dots, \frac{r_{i-1}}{d_i}\right)$. According to \cite{ab2}, the polynomial $f=g_{h+1}$ has one place at infinity if and only if the following conditions hold:

\begin{enumerate}

\item for all $1\leq i\leq h-1, r_id_i >r_{i+1}d_{i+1}$,

\item for all $1\leq i\leq h, g_{i+1}$ is straight with respect to $(g_i,\underline{r}^{i-1})$.
\end{enumerate}

\begin{exemples} The implementation of \texttt{SemigroupOfValuesOfPlaneCurveWithSinglePlaceAt\-Infi\-nity} contains this criterion.
\begin{enumerate}[i)]

\item Let $f(x,y)=(y^3-x^2)^2-y$. We have $r_0=d_1=6,r_1=4,d_2=2, g_1=y,g_2=y^3-x^2$. Now $r_2=\mathrm{fint}(g_2,y,(3,2))=2$, hence $d_3=2$, consequently $f$ has at least two places at infinity.
\begin{verbatim}
gap> SemigroupOfValuesOfPlaneCurveWithSinglePlaceAtInfinity((y^3-x^2)^2-y);
Error, Error the polynomial is not irreductible or it has not a single place
at infinity called from 
...
\end{verbatim}

\item Let $f(x,y)=(y^3-x^2)^2-x^5y$. We have $r_0=d_1=6,r_1=4,d_2=2, g_1=y,g_2=y^3-x^2$. Now $r_2=\mathrm{fint}(g_2,y,(3,2))=17$, hence $d_3=1$.  Furthermore, the straightness condition is satisfied. However, $r_1d_1=24<r_2d_2=34$, hence $f$ has at least two places at infinity.
\begin{verbatim}
gap> SemigroupOfValuesOfPlaneCurveWithSinglePlaceAtInfinity((y^3-x^2)^2-x^5*y);
Error, The polynomial does not have a single place at infinity or the leading 
coefficient in x is not a rational number called from
...
\end{verbatim}

\item Let $f(x,y)=y^5-x^4+x^4y$.  We have $r_0=d_1=5, r_1=4, d_2=1$, but $f$ is not straight with respect to $(y,(5,4))$ because $\mathrm{fint}(f,x^4, (5,4))=20>4r_1=16$, hence $f$ has at least two places at infinity.
\begin{verbatim}
gap> SemigroupOfValuesOfPlaneCurveWithSinglePlaceAtInfinity(y^5-x^4+x^4*y);
Error, The polynomial does not have a single place at infinity or the leading 
coefficient in x is not a rational number called from sv( arg[1] ) called from
...
\end{verbatim}

\item Let $f(x,y)=((y^3-x^2)^2-xy)^2-(y^3-x^2)$. We have $r_0=d_1=12$, $r_1=8$, $d_2=4$, $g_1=y$, $g_2=y^3-x^2$. Now 
$r_2=\mathrm{fint}(g_2,y,(3,2))=10$, hence $d_3=2$ and  $g_3=(y^3-x^2)^2-xy$. Now $r_3=\mathrm{fint}(g_3,(y^3-x^2,y),(6,4,5))=5$, hence $d_3=1$. Furthermore the straightness condition is satisfied for $g_2$, $g_3$ and $f$. Since $r_1d_1>r_2d_2>r_3d_3$, we deduce that  $f$ has one place at infinity and $\Gamma(f)=\langle 12,8,10,5\rangle$.
\begin{verbatim}
gap> SemigroupOfValuesOfPlaneCurveWithSinglePlaceAtInfinity(f,"all");
[ [ 12, 8, 10, 5 ], [ y, y^3-x^2, y^6-2*x^2*y^3+x^4-x*y ] ]
\end{verbatim}
\end{enumerate}
\end{exemples}

\section{One place curves with a fixed genus}\label{sec4}

Let $f=y^{n}+a_1(x)y^{n-1}+\dots+a_n(x)$ be a polynomial with one place at infinity and assume that $a_1(x)=0$ and also that $\deg_x(a_i(x))<i$ for all $2\leq i\leq n$. Let $r_0=n$, $r_1=\deg_xa_n(x)$, $r_2,\dots,r_h$ be the set of generators of $\Gamma(f)$ constructed as in Section \ref{sec1}. Let $d_1=n$ and for all $1\leq k\leq h$, let $d_{k+1}=\gcd(r_k,d_k)$ and $e_k=\frac{d_k}{d_{k+1}}$. We have
\begin{enumerate}[1)]
\item for all $1\leq k\leq h-1, r_kd_k>r_{k+1}d_{k+1}$,

\item for all $1\leq k\leq h, e_kr_k\in \langle r_0,\dots,r_{k-1}\rangle$.
\end{enumerate}
If furthermore $n=r_0=d_1>r_1>d_2>\dots >d_{h+1}=1$, then $\mu(\mathrm{App}_{d_2}(f))>0$ and $(r_0,\dots,r_h)$ is a $\delta$-sequence.  Conversely,  given a sequence of coprime integers $\underline{r}=(r_0,r_1,\dots,r_h)\in\NN$, if  $d_1=n$ and $d_{k+1}=\gcd(r_k,d_k), e_k=\frac{d_k}{d_{k+1}}$ for all $1\leq k\leq h$ and if Conditions 1) and 2) above are fulfilled, then there exists a polynomial $f$ with one place at infinity such that $\Gamma(f)=\langle r_0,\dots,r_h\rangle$. If furthermore $r_0=d_1>r_1>d_2>\dots >d_{h+1}=1$, then there exists  a polynomial $f=y^{r_0}+a_2(x)y^{r_0-2}+\dots+a_{r_0}(x)$ with $\mathrm{rdeg}(f)=r_0$, $\deg_xa_{r_0}(x)=r_1$, and $\Gamma(f)=\langle r_0,\dots,r_h\rangle$. The straightness of the set of generalized Newton polygons gives us that set all such polynomials. Set $n=r_0$ and $r_1=m$ and assume that $n>m$. We have the following algorithmic construction of these polynomials.

\begin{itemize}
\item $h=1$: $f(x,y)=y^{n}+a_1x^{r_1}+\sum_{ni+r_1j<nm}c_{ij}x^iy^j$, where $a_1\in\KK^*$.

\item $h>1$: Let $\underline{r}^ h=\left(\frac{r_0}{d_h}, \frac{r_1}{d_h},\dots,\frac{r_{h-1}}{d_h}\right)$. The sequence $\underline{r}^ h$ satisfies the same conditions as $\underline{r}$. Let $g$ be a polynomial with one place at infinity such that $\mathrm{rdeg}(g)=$ $\deg_yg=\frac{n}{d_h}$, $\deg_xg=\frac{r_1}{d_h}$, and $\Gamma(g)=\langle \underline{r}^h\rangle$. Let $B=\lbrace \underline{\theta}=(\theta_0,\dots,\theta_{h-1}) \mid \hbox{for all } i=1,\dots,h-1,\theta_i<e_i\rbrace$ and let $\underline{\theta}^0$ be the unique element of $B$ such that $r_hd_h=\sum_{i=0}^{h-1}\theta_i^0r_i$. We set
$$
f=g_h^{d_h}+\alpha_2(x,y)g_h^{d_h-2}+\dots+\alpha_{d_h}(x,y)+a_hx^{\theta_0^0}g_1^{\theta_1^0}\dots g_{h-1}^{\theta_{h-1}^0}, 
$$
where for all $2\leq i\leq d_h$, if $\alpha_i(x,y)\not=0$ then $\deg_y\alpha_i(x,y)<\frac{n}{d_h}$. Furthermore, write $\alpha_i(x,y)= \sum_{\underline{\theta}\in B}c_{\underline{\theta}}(x)g_1^{\theta_1}\dots g_{h-1}^{\theta_{h-1}}$. If $c_{\underline{\theta}}(x)\not=0$, then 
$\deg_xc_{\underline{\theta}}(x)r_0+\theta_1r_1+\dots+\theta_{h-1}r_{h-1}<ir_h$.
\end{itemize}

\begin{example}
Let us compute a curve associated to the $\delta$-sequence $(6,4,3)$. In order to see the recursive process, we set the information level of the \texttt{numericalsgps} package to 2.
\begin{verbatim}
gap> SetInfoLevel(InfoNumSgps,2);
gap> CurveAssociatedToDeltaSequence([6,4,3]);
#I  Temporal curve: y^3-x^2
#I  Temporal curve: y^6-2*x^2*y^3+x^4-x
y^6-2*x^2*y^3+x^4-x
\end{verbatim}
\end{example}

\section{Abhyankar semigroups with a given genus}\label{sec5}

Let $g$ be a  positive integer and let $\mu=2g$. There exists a numerical symmetric semigroup $\Gamma$ and a system of generators of $\Gamma$ such that the following conditions hold:

\begin{enumerate}

\item $\mu$ is the conductor $\mathrm C(\Gamma)$ of $\Gamma$,

\item $\Gamma=\langle r_0,r_1,\dots,r_h\rangle$, $r_0 > r_1$, and $(r_0,\dots,r_1)$ is a $\delta$-sequence.
\end{enumerate}

In particular, $\Gamma=\Gamma(f)$, where $f\in\KK[x,y]$ is a polynomial with one place at infinity such that $\mathrm{rdeg}(f)=r_0$. The simplest example of such a semigroup if $\langle \mu+1,2\rangle$. Note that if $\Gamma$ is such a semigroup, then $\mu\geq 2(2^h-1)$ (see \cite[Proposition 6.7]{a-1}). In particular $h\leq b=\log_2(g+1)$. Let $1\leq h\leq b$. 

In light of Proposition \ref{recursive-appr}, if we denote by $\mu_h=\mathrm{C}(\langle r_0/d_h,\ldots, r_{h-1}/d_h\rangle)$, then $\mu=d_h\mu_h+(d_h-1)r_h-d_h+1 =d_h\mu_h+(r_h-1)(d_h-1)$ and $\gcd(r_h,d_h)=d_{h+1}=1$. 
\begin{itemize}
\item If $h=1$, then $\mu=(r_0-1)(r_1-1)$. Since $r_0>r_1$, $(r_1-1)^2< \mu$. Hence $r_1< \sqrt{\mu}+1$, which is less than or equal to $\mu-1$ for all $\mu\ge 4$. Notice that $\mu$ cannot be $3$, and the case $\mu=2$ is $\langle 3,2\rangle$.

\item For $h>1$, $d_h\mu_h>1$ and $d_h\ge 2$. Consequently, $r_h-1\le (r_h-1)(d_h-1)=\mu-d_h\mu_h<\mu-1$, whence $r_h<\mu$.
\end{itemize}
Observe also that $\mu-1$ is the Frobenius number of $\Gamma$, and so $r_h\neq \mu-1$. This implies that for $\Gamma\neq \langle 3,2\rangle$,
\[2\le r_h\le \mu-2.\]
Thus there are finitely many possible $r_h$, and for each of these $r_h$, $(d_h-1)(r_h-1)\le (d_h-1)(r_h-1)+d_h\mu_h=\mu$. Hence 
\[2\le d_h\le\frac{\mu}{r_h-1}+1.\]
For each pair $r_h$ and $d_h$ we find recursively the sequences $\langle r_0',\ldots, r_{h-1}'\rangle$ with conductor $\mu_h$. This was the idea used to implement \texttt{DeltaSequencesWithFrobeniusNumber}.

\begin{nota}
From Lemma \ref{form-c}, $\mathrm{int}(f,f_y)=\mu+r_0-1\leq r_0(r_0-1)$ by B\'ezout's Theorem. Hence $\mu\leq (r_0-1)^2$, which implies  that $r_0\geq \sqrt{\mu}+1$.
\end{nota}

\begin{example}
It may happen that several $\delta$-sequences generate the same numerical semigroup. This is why in general there are more $\delta$-sequences than Abhyankar semigroups for a fixed genus.
\begin{verbatim}
gap> l:=DeltaSequencesWithFrobeniusNumber(13);
[ [ 6, 4, 11 ], [ 8, 3 ], [ 8, 6, 3 ], [ 9, 6, 5 ], [ 10, 4, 7 ], [ 12, 8, 3 ],
  [ 12, 8, 6, 3 ], [ 15, 2 ], [ 15, 6, 2 ], [ 15, 10, 2 ] ]
gap> Length(l);
10
gap> Length(Set(l,NumericalSemigroup));
5
\end{verbatim}

Next figure plots the number of Abhyankar semigroups and $\delta$-sequences.

\includegraphics[scale=.6]{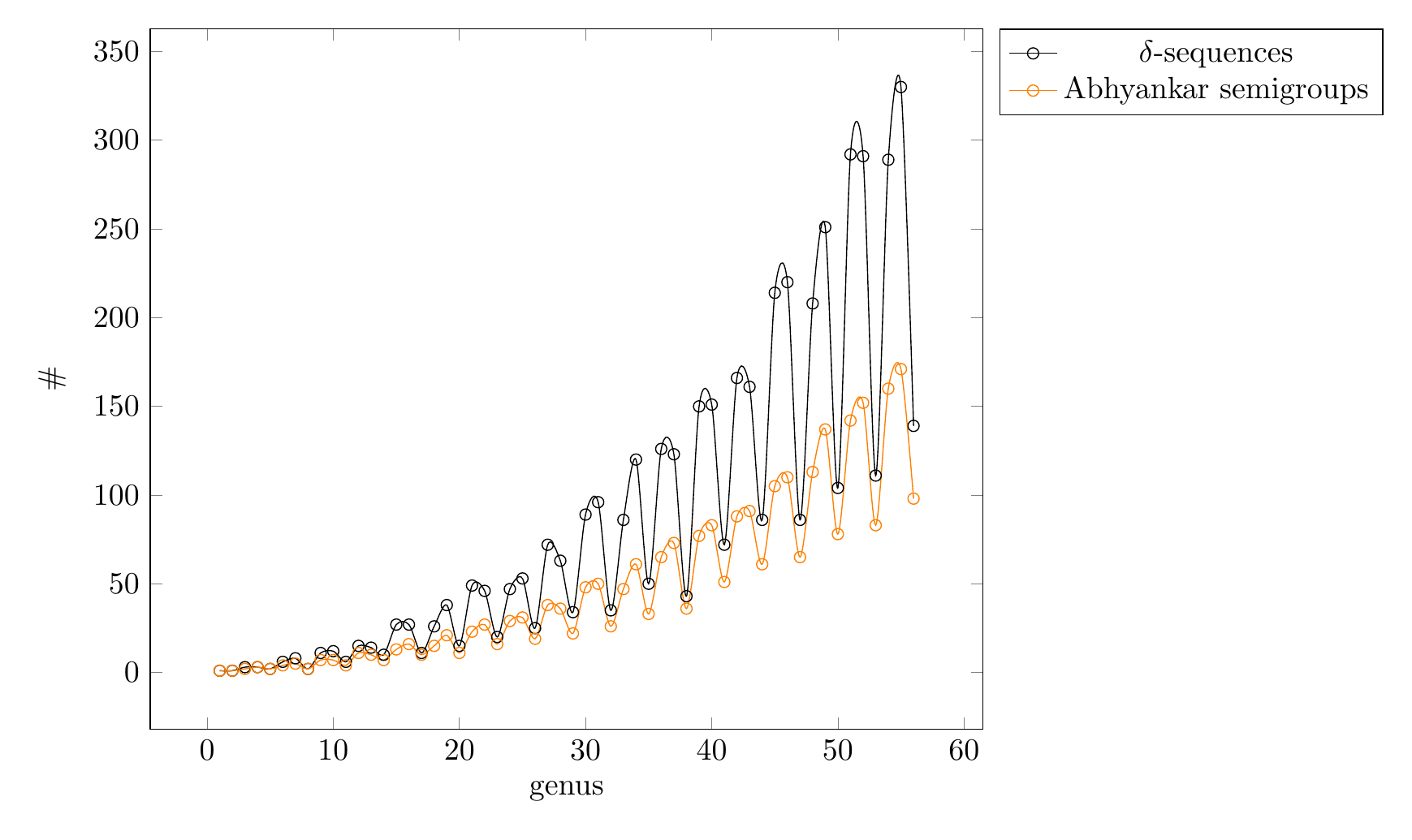} 

The following plotting  compares the number of Abhyankar semigroups with given genus with other well known families of complete intersection numerical semigroups.

\includegraphics[scale=.9]{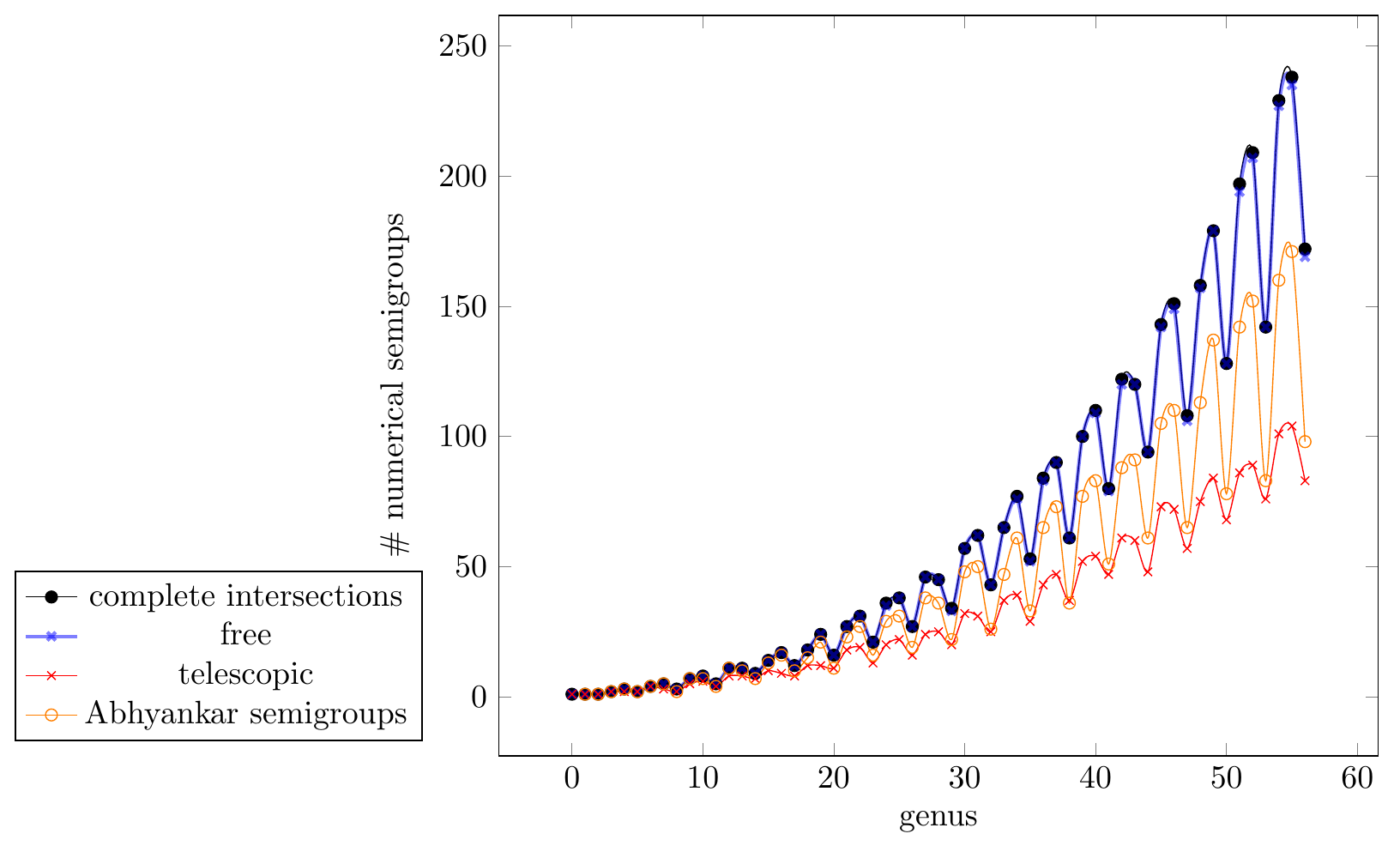}

\end{example}

Let $\Gamma$ be a semigroup with conductor $\mu=2g$ and let $f$ be a polynomial with one place at infinity such that $\Gamma(f)=\Gamma$. Let $(f-\lambda)_{\lambda\in\KK}$ be the pencil of polynomials defined by $f$. Given an  element $F=f-\lambda_0$ of the pencil, if $\mu_{\infty}$ denotes the local Milnor number of $F$ at the point at infinity, then we have $\mu+\mu_{\infty}=(r_0-1)(r_0-2)$. Also, if $g$ denotes the genus of the curve $\mathrm V(F)$, then $2g+\sum_{p\in \mathrm V(F)}\mu_p+r_p-1+\mu_{\infty}=(r_0-1)(r_0-2)$, where $\mu_p$ (respectively $r_p$) denotes the Milnor number (respectively the number of places) at $p$. If the curve $\mathrm V(F)$ is nonsingular in $\KK^2$, then $2g+\mu_{\infty}=(r_0-1)(r_0-2)=\mu+\mu_{\infty}$, hence $\mu=2g$. In particular $g$ is the geometric genus of a nonsingular element of the pencil $(f-\lambda)_{\lambda\in\KK}$.

\section{Embedding of one place curves in the affine plane}\label{sec6}

Let $f(x,y)$ be a nonzero polynomial of $\KK[x,y]$ and let $A(f)={{\KK[x,y]}/ {(f)}}$. Given another polynomial $g$, we say that $f$ and $g$ are isomorphic if the two $\KK$-algebras $A(f)$ and  $A(g)$ are isomorphic. We say that $f$ is equivalent to $g$ if $g=\sigma(f)$ for some automorphism $\sigma$ of $\KK[x,y]$. Then the following natural question arises: which isomorphic polynomials are equivalent?

Suppose that $f$ has one place at infinity and let $g$ be a polynomial with one place at infinity. It follows from Theorem \ref{th-equiv} that if $f$ is equivalent  to $g$,  then $f$ and $g$ have the same $\delta$-sequence. The converse being not true in general.

Let $x(t)=t^n+c_1t^{n-1}+\dots+c_n, y(t)=t^m+c'_1t^{m-1}+\dots+c'_m$ be two polynomials of $\KK[t]$. Suppose that $n>m>\gcd(n,m)$ and let $f(x,y)=y^n+a_1(x)y^{n-1}+\dots+a_n(x)$ be the $y$-monic generator of the kernel of the map $\phi:\KK[x,y]\to \KK[t], \phi(x)=x(t)$ and $\phi(y)=y(t)$ (there is such a generator because the curve has a single place at infinity). Then $f$ has one place at infinity and $A(f)\cong \KK[x(t),y(t)]$. Let $g$ be another curve with one place at infinity. Then $f$ is isomorphic to $g$ if and only if $A(g)\cong \KK[x_1(t),y_1(t))]\cong \KK[x(t),y(t)]$.  We have the following result.

\begin{proposicion}\label{equiv}
Let $x(t),y(t),x_1(t),y_1(t)\in\KK[t]$ and let $f$ (respectively $g$) be the $y$-monic generator of the map $\phi:\KK[x,y]\to \KK[t], \phi(x)=x(t)$ and $\phi(y)=y(t)$ (respectively $\psi:\KK[x,y]\to \KK[t], \psi(x)=x_1(t), \psi(y)=y_1(t)$). If $A(f)\cong \KK[x(t),y(t)]\cong A(g)\cong \KK[x_1(t),y_1(t)]$, then  $\Gamma(f)=\Gamma(g)$.
\end{proposicion}
\begin{proof} 
The fact that $x(t),y(t)\in \KK[x_1(t),y_1(t)]$ implies that there exist $P(x,y),Q(x,y)$ such that $x(t)=P(x_1(t),y_1(t))$ and $y(t)=Q(x_1(t),y_1(t))$.  Let $a\in \Gamma(f)$ and let $h\in\KK[x,y]$ be such that $\mathrm{int}(f,h)=\deg_t h(x(t),y(t))=a$. If $h_1(x,y)=h(P(x,y),Q(x,y))$, then \[h_1(x_1(t),y_1(t)) =h(P(x_1(t),y_1(t)),Q(x_1(t),y_1(t))= h(x(t),y(t))\] and $\deg_th_1(x_1(t),y_1(t))=a=\mathrm{int} (g,h_1)$, whence $a\in\Gamma(g)$. We prove in a similar way that $\Gamma(g)$ is contained in $\Gamma(f)$.
\end{proof}

\begin{example} 
Isomorphic does not imply equivalent for polynomial curves. Let $z(t)=t^3-a,a\not=0, x(t)=tz(t)=t^4-at$, and $y(t)=z(t)^2+\frac{a}{2}z(t)$. We have $x(t)^3=t^3z(t)^3=z(t)^4+az(t)^3$. We shall prove that $\KK[x(t),y(t)]=\KK[x(t),z(t)]$. The first inclusion is clear. Also, $y(t)^2=z(t)^4+az(t)^2+\frac{a^2}{4}z(t)^2=x(t)^3+\frac{a^2}{4}y(t)-\frac{a^3}{8}z(t)$, whence $\frac{a^3}{8}z(t)=y(t)^2-x(t)^3-\frac{a^2}{4}y(t)\in\KK[x(t),y(t)]$.  Let $f$ (respectively $g$) be the $y$-monic generator of the map $\phi:\KK[x,y]\to \KK[t]$, $\phi(x)=x(t)$ and $\phi(y)=y(t)$ (respectively $\psi:\KK[x,y]\to \KK[t]$, $\psi(x)=x(t)$, $\psi(y)=z(t)$). Then $f(x,y)=\frac{1}{2}a^2x^3y-\frac{1}{2}a^2y^3+x^6-2x^3y^2+y^4+\frac{1}{16}a^4x^3=\left( y^3-x^2-\frac{1}{4}a^2\right)^2+ \frac{1}{2}a^2x^3y+\frac{1}{16}a^4x^3-\frac{1}{2}a^2x^2-\frac{1}{16}a^4$ and $g(x,y)=y^4+ay^3-x^3$. Hence the $\delta$-sequence of $f$ (resp. $g$) is $\langle 6,4,3\rangle$ (resp. $\langle 4,3\rangle$). By Theorem \ref{th-equiv}, $f$ is not equivalent to $g$.
\end{example}

This example provides a counter example for the conjecture stated in \cite{y}.
 
\begin{nota}
\begin{enumerate}[(i)]
\item  Let $f(x,y)=1-xy$ and $g(x,y)=1-xy^k, k\geq 2$. Let $C_1:f=0$, $C_k:g=0$, and let $\phi: C_1\to C_k, \phi(a,b)=(a^k,b)$ and $\psi: C_k\to C_1, \psi(c,d)=(cd^{k-1},d)$. We have $\psi\circ\phi(a,b)=\psi(a^k,b)=(a^kb^{k-1},b)=((ab)^{k-1}a,b)=(a,b)$ and $\phi\circ\psi(c,d)=\phi(cd^{k-1},d)=(c^kd^{k(k-1)},d)=(c(cd^k)^{k-1},d)=(c,d)$. Hence $C_1$ and $C_k$ are isomorphic. Clearly $C_1$ and $C_k$ are not equivalent. Thus we have infinitely many inequivalent embeddings of $C_1$ in the affine plane.  Note that $C_1$ (respectively $C_k$) has two places at infinity.

\item Let $f(x,y)$ be a plane curve with one place at infinity and let $\mathrm r(f)=y^n+a_2(x)y^{n-2}+\dots+a_n(x)$. If $g$ is another one place curve such that $A(f)$ is isomorphic to $A(g)$, then $\Gamma(f)=\Gamma(g)$. Hence the number of nonequivalent embeddings of $f$ is bounded by the number of semigroups $\bar{\Gamma}=\langle r_0,r_1,\dots,r_h\rangle$ such that:

\begin{enumerate}[(1)]
\item  $\bar{\Gamma}=\Gamma(f)$,

\item $\bar{\Gamma}=\langle r_0,r_1,\dots,r_h\rangle$ and $(r_0,r_1,\dots,r_h)$ is a $\delta$-sequence.
\end{enumerate} 
\end{enumerate}
\end{nota}

\begin{example}
Let $f(x,y)$ be such that $\Gamma(f)=\langle 7,6\rangle$. We have $\mu(f)=30$. Let $h=2$ and let $\bar{\Gamma}=\langle r_0,r_1,r_2\rangle$ with the Properties (1) and  (2). Since $7$ is a prime number, $r_0>7$. Hence $r_0\in \langle 7,6\rangle$. Furthermore, $r_2$ is either $6$ or $7$. If $r_2=6$, then $d_2=7$. Thus $\mu(f)=30=\mu_2d_2+(d_2-1)(r_2-1)=7\mu_2+30$, and $\mu_2=0$. In particular $(r_0,7,6)$ is not reduced. Finally $r_2=7$ and $d_2=2,3$. This gives us the following solutions: $\delta_1=( 14,6,7)$, and $\delta_2=(21,6,7)$. Note that if $f(x,y)=y^7-x^6-x$, $g(x,y)=(y^7-x^3)^2-x$, $h(x,y)=(y^7-x^2)^3-x$, then $A(f), A(g)$, and $A(h)$ are isomorphic (though they are not equivalent since their associated $\delta$-sequences are different).

We can perform this task with our \texttt{GAP} implementation.
\begin{verbatim}
gap> s:=NumericalSemigroup(6,7);
<Modular numerical semigroup satisfying 7x mod 42 <= x >
gap> FrobeniusNumber(s);
29
gap> DeltaSequencesWithFrobeniusNumber(29);
[ [ 7, 6 ], [ 8, 6, 19 ], [ 9, 6, 13 ], [ 10, 6, 15 ], [ 11, 4 ], [ 12, 8, 10, 15 ], 
  [ 12, 8, 14, 11 ], [ 12, 9, 7 ], [ 14, 4, 19 ], [ 14, 6, 7 ], [ 15, 6, 10 ],
  [ 15, 10, 6 ], [ 16, 3 ], [ 16, 6, 3 ], [ 16, 12, 3 ],  [ 16, 12, 6, 3 ], 
  [ 18, 4, 15 ], [ 18, 12, 9, 7 ],[ 18, 12, 15, 4 ], [ 21, 6, 7 ], [ 22, 4, 11 ],
  [ 24, 16, 3 ], [ 24, 16, 6, 3 ], [ 24, 16, 12, 3 ],  [ 24, 16, 12, 6, 3 ], 
  [ 27, 6, 4 ], [ 31, 2 ] ]
gap> Filtered(last, gs->NumericalSemigroup(gs)=s);
[ [ 7, 6 ], [ 14, 6, 7 ], [ 21, 6, 7 ] ]
\end{verbatim}
In general, given $f$, this procedure gives us the $\delta$-sequences of the candidates $g$ such that $A(f)\cong A(g)$. 
\end{example}

\section*{Acknowledgements}
The authors would like to thank the Departamento de \'Algebra of the Universidad de Granada, and the Laboratoire de Math\'ematiques of the Universit\'e d'Angers, respectively, for their kind hospitality.

\end{document}